\documentclass[11pt]{amsart}

\usepackage{amsmath}
\usepackage{amssymb}
\usepackage{amsthm}
\usepackage{amsbsy}
\usepackage{amsfonts}

\usepackage[bitstream-charter]{mathdesign}
\usepackage[T1]{fontenc}

\newtheorem{thm}{Theorem}
\newtheorem{lem}[thm]{Lemma}

\newtheorem{defn}{Definition}
\newtheorem{rem}[thm]{Remark}
\newtheorem{nota}{Notation}

\usepackage{esint}

\newcommand{\R}{\mathbb{R}}

\newcommand{\bL}{\boldsymbol L}

\newcommand{\bW}{\boldsymbol W}

\newcommand{\bbf}{\boldsymbol f}
\newcommand{\bg}{\boldsymbol g}
\newcommand{\bv}{\boldsymbol v}

\newcommand{\x}{\boldsymbol x}
\newcommand{\y}{\boldsymbol y}
\newcommand{\z}{\boldsymbol z}

\newcommand{\cA}{\mathcal{A}}
\newcommand{\cB}{\mathcal{B}}

\newcommand{\cD}{\mathcal{D}}

\newcommand{\set}[1]{\{#1\}}
\newcommand{\Set}[1]{\Big\{#1\Big\}}
\newcommand{\inn}[1]{\langle#1\rangle}
\newcommand{\norm}[1]{\Vert#1\Vert}

\begin{document}
\baselineskip=18pt

\title[Regularity criterion for the Navier--Stokes equations]{A new local regularity criterion for suitable weak solutions of the Navier--Stokes equations in terms of the velocity gradient}

\author{Hi Jun Choe \& Joerg Wolf \& Minsuk Yang}

\address{H. J. Choe: Department of Mathematics, 
Yonsei University,
Yonseiro 50, Seodaemungu Seoul, 
Republic of Korea}
\email{choe@yonsei.ac.kr}

\address{J. Wolf: Department of Mathematics,
Humboldt University Berlin,
Unter den Linden 6, 10099 Berlin,
Germany}
\email{jwolf@math.hu-berlin.de}

\address{M. Yang: Korea Institute for Advanced Study \\
Hoegiro 85, Dongdaemungu Seoul,
Republic of Korea}

\email{yangm@kias.re.kr}

\begin{abstract}
We study the partial regularity of suitable weak solutions to the three dimensional incompressible Navier--Stokes equations.
There have been several attempts to refine the Caffarelli--Kohn--Nirenberg criterion (1982).
We present an improved version of the CKN criterion with a direct method, which also provides the quantitative relation in Seregin's criterion (2007). 
\end{abstract}

\maketitle

\section{Introduction}
\label{S1}

We consider the Navier--Stokes equations
\begin{equation}
\label{E11}
\begin{split}
(\partial_t - \Delta) \bv + (\bv \cdot \nabla) \bv + \nabla p = \bbf \quad &\text{ in } \Omega\times(0,T) \\
\nabla \cdot \bv = 0 \quad &\text{ in } \Omega\times(0,T)
\end{split}
\end{equation}
where $\Omega \subset \R^3$ is a bounded domain with $C^2$ boundary and $T>0$.
The state variables $\bv$ and $p$ denote the velocity field of the fluid and its pressure. 
We complete the above equations by the following boundary and initial conditions
\begin{align*}
\bv &= 0 \quad \text{ on } \partial\Omega\times(0,T) \\
\bv &= \bv_0 \quad \text{in } \Omega\times\set{0}
\end{align*}
where the initial velocity $\bv_0$ is sufficiently regular. 
Throughout this paper, we assume that $(\bv,p)$ is a suitable weak solution to this problem and the definition will be given in the next section.

There are a huge number of important papers that contribute to the regularity problem of suitable weak solutions to the Navier--Stokes equations and there are many good survey papers and books.
So, we only mention a few of them.
Scheffer \cite{Sc76,Sc77} introduced partial regularity for the Navier--Stokes system.
Caffarelli, Kohn and Nirenberg \cite{CKN82} further strengthened Scheffer's results.
Lin \cite{Li98} gave a new short proof by an indirect argument.
Neustupa \cite{Ne99} and Ladyzhenskaya and Seregin \cite{LS99} investigated partial regularity. 
Choe and Lewis \cite{CL00} studied singular set by using a generalized Hausdorff measure.
Escauriaza, Seregin, and \v{S}ver\'ak \cite{ESS03} proved the marginal case of the so-called Ladyzhenskaya--Prodi--Serrin condition based on the unique continuation theory for parabolic equations.
Gustafson, Kang, and Tsai \cite{GKT07} generalize several previously known criteria.

Among the many important regularity conditions, the following criterion plays an important role because it gives better information about the possible singular points: 
There exists an absolute positive constant $\epsilon$ such that $\z = (\x,t) \in \Omega\times(0,T)$ is a regular point if
\begin{equation}
\label{E13}
\limsup_{r\to0} r^{-1} \iint_{Q(\z,r)} |\nabla \bv|^2 d\y ds < \epsilon
\end{equation}
where $Q(\z,r)$ denotes the parabolic cylinder $B(\x,r) \times (t-r^2,t) \subset \R^3 \times \R$.

There have been several attempts to refine this criterion.
In particular, Seregin \cite{Se07} weaken the above condition as follows:
for each $0<M<\infty$ there exists a positive number $\epsilon(M)$ such that $\z \in \Omega\times(0,T)$ is a regular point if 
\begin{equation}
\label{E15}
\begin{split}
&\limsup_{r\to0} r^{-1} \iint_{Q(\z,r)} |\nabla \bv|^2 d\y ds \le M \\
&\liminf_{r\to0} r^{-1} \iint_{Q(\z,r)} |\nabla \bv|^2 d\y ds < \epsilon(M).
\end{split}
\end{equation}
The proof was done by an indirect argument, which has been widely used as an effective way to prove such kind of regularity theorems in the field of nonlinear PDEs.
The proof goes as follows.
If the theorem is false, then there should exist a sequence of suitable solutions such that the scaled quantity 
\[
r^{-1} \iint_{Q(\z,r)} |\nabla \bv_n|^2 d\y ds
\]
tends to zero on a fixed particular cylinder centered at a singular point $\z$.
One can show that the uniform boundedness occurs to ensure a compactness lemma and its sub-sequential limit must be regular enough at the point $\z$, wihch gives a contradiction to the fact that $\z$ is a singular point.
By this argument one can know the theorem is true so that $\epsilon(M)$ should exist.
However, the argument does not provide any specific information about $\epsilon(M)$, even the quantitative dependence on $M$ is unclear.

In this paper, we shall give a new refined local regularity criterion of suitable weak solutions to the Navier--Stokes system with a direct iteration method so that our theorem shows a reverse relation between $M$ and $\epsilon(M)$ and gives at least a quantitative upper bound of $\epsilon(M)$ in terms of $M$.
For simplicity we use the following notation.

\begin{defn}
For $9/5 \le q \le 2$, we define 
\[E_q(\z,r) = r^{-5+2q} \iint_{Q(\z,r)} |\nabla \bv|^q d\y ds\]
and denote 
\[\overline{E}_q(\z) = \limsup_{r\to0} E_q(\z,r) \quad \text{ and } \quad 
\underline{E}_q(\z) = \liminf_{r\to0} E_q(\z,r)\]
We omit the subscript $q$ when $q=2$.
\end{defn}

Here are our main results.

\begin{thm}
\label{T1}
Let $9/5 \le q < 2$ and $\bbf=0$.
There exists a positive number $\epsilon$ such that $\z \in \Omega\times(0,T)$ is a regular point if 
\[\overline{E}_q(\z)^{(5-q)/(q-1)} \underline{E}_q(\z) < \epsilon.\]
\end{thm}

The range $9/5 \le q \le 2$ is essential in view of our interpolation inequalities and the endpoint exponent $9/5$ is important when one deals with a reverse H\"older-type inequality.
But, the restriction $\bbf=0$ is inessential.
Actually, under some mild integrability condition on $\bbf$, one can easily show that the contribution from $\bbf$ is small enough so that the theorem is still true for nonzero forces $\bbf$.

We have a further improvement when $q=2$.
In this case, we treat $\bbf \neq 0$ as an illustration how to control the nonzero forces.

\begin{thm}
\label{T2}
Let $\bbf \in L^{r}(\Omega _T)$ for some $r > 5/2$.
There exists a positive number $\epsilon$ such that $\z \in \Omega\times(0,T)$ is a regular point if 
\[\overline{E}(\z) \underline{E}(\z) < \epsilon.\]
\end{thm}   

This is a quantitative version of \eqref{E15}: the point $\z \in \Omega\times(0,T)$ is regular if 
\[
\underline{E}(\z) < \frac{\epsilon}{M}. 
\]

\begin{rem}
We shall define several scaled functionals and give various relations among them. 
However, the estimates of those functionals in this paper will not depend on the reference point $\z$.
So, we shall assume $\z=(0,0)$ and $Q(z,2) \subset \Omega\times(-8,8)$ for notational convenience.
From now, we suppress $\z$.
\end{rem}

\section{Preliminaries}
\label{S2}

We denote by $L^p(\Omega)$ and $W^{k,p}(\Omega)$ the standard Lebesgue and Sobolev spaces and
we use the boldface letters for the space of vector or tensor fields.
We denote by $\cD_\sigma(\Omega)$ the set of all solenoidal vector fields $\phi \in C_c^\infty(\Omega)$.
We define $\bL_\sigma^2(\Omega)$ to be the closure of $\cD_\sigma(\Omega)$ in $\bL^2(\Omega)$ and $\bW_\sigma^{1,2}(\Omega)$ to be the closure of $\cD_\sigma(\Omega)$ in $\bW^{1,2}(\Omega)$.

\begin{defn}[suitable weak solutions]
Let $\Omega_T=\Omega\times(0,T)$.
Suppose that $\bbf \in L^p(\Omega_T)$ for some $p>5/2$.
We say that $(\bv, p)$ is a suitable weak solution to \eqref{E11} if 
\[\bv \in L^\infty(0,T; \bL_\sigma^2(\Omega)) \cap L^2(0,T; \bW_\sigma^{1,2}(\Omega)), \qquad p \in L^{3/2}(\Omega_T),
\]
and $(\bv, p)$ solves the Navier--Stokes equations in $\Omega_T$ in the sense of distributions and satisfies the generalized energy inequality
\begin{equation}
\label{E21}
\begin{split}
&\int_{\Omega} |\bv(t)|^2 \phi(t) d\x + 2 \int_0^t \int_{\Omega} |\nabla \bv|^2 \phi d\x ds \\
&\le \int_0^t \int_{\Omega} |\bv|^2 (\partial_t \phi + \Delta \phi) d\x ds 
+ \int_0^t \int_{\Omega} |\bv|^2 \bv \cdot \nabla \phi d\x ds \\
&\quad + 2\int_0^t \int_{\Omega} p \bv \cdot \nabla \phi d\x ds 
+ 2\int_0^t \int_{\Omega} \bbf \cdot \bv \phi d\x ds 
\end{split}
\end{equation}
for almost all $t \in (0,T)$ and for all nonnegative $\phi \in C_c^\infty(\Omega_T)$.
\end{defn}

Throughout the paper, we use the following notation.

\begin{nota}
We denote the average value of $g$ over the set $E$ by 
\[\inn{g}_E = \fint_E g d\mu = \mu(E)^{-1} \int_E g d\mu.\]
We denote $A \lesssim B$ if there exists a generic positive constant $C$ such that $|A| \le CB$.
\end{nota}

\section{Local energy inequalities}
\label{S3}

We shall define several scaled functionals to describe neatly various relations among them.
The aim of this section is to present local Caccioppoli-type inequalities.

\begin{defn}
[scaled functionals I]
Let 
\begin{align*}
A(r) &= r^{-1} \sup_{t-r^2<s<t} \int_{B(\x,r)} |\bv|^2 d\y \\
C(r) &= r^{-2} \iint_{Q(r)} |\bv|^3 d\y ds\\
D(r) &= r^{-2} \iint_{Q(r)} |p-\inn{p}_{B(r)}|^{3/2} d\y ds
\end{align*}
where $\inn{p}_{B(r)} = \fint_{B(r)} p d\y$.
\end{defn}

From the definition of suitable weak solution we get the next lemma.
Indeed, it is a direct consequence of the inequality \eqref{E21} with a standard cutoff function $\phi$, so we omit its proof.

\begin{lem}
[local energy inequality I]
\label{L31}
For $0 < r \le 1$
\[A(r) + E(r) \lesssim C(2r)^{2/3} + C(2r) + C(2r)^{1/3} D(2r)^{2/3}.\]
\end{lem}

In terms of the following scaled functionals, we shall derive another version of a local  Caccioppoli-type inequality.

\begin{defn}
[scaled functionals II]
Let 
\begin{align*}
G(r) &= r^{-1} \int_{-r^2}^0 \Big(\int_{B(r)} |\bv|^6 d\y\Big)^{1/3} ds \\
P(r) &= r^{-2} \inf_{c\in\R} \left(\int_{-r^2}^0 \Big(\int_{B(r)} |p-c|^3 d\y\Big)^{1/3} ds\right)^2.
\end{align*}
\end{defn}

\begin{lem}
[local energy inequality II]
\label{L32}
For $0 < r \le 1$
\[A(r) + E(r) \lesssim [1+E(2r)] G(2r) + P(2r).\]
\end{lem}

\begin{proof}
First, we fix $\phi \in C^{\infty}_{\rm c}(\Omega_T)$ satisfying $0 \le \phi \le 1$ in $\R^3$, 
$\phi \equiv 1$ on $Q(r)$, $\phi \equiv 0$ in $\R^3 \times (-\infty,0) \setminus Q(2r)^c$ and
\[
|\partial_t \phi| + |\nabla^2\phi| + |\nabla \phi|^2 \lesssim r^{-2}.
\]
Then, by the definition of the suitable weak solution, we have 
\begin{equation}
\label{E31}
\begin{split}
&\int |\bv(t)|^2 \phi^2 dy
+ \iint |\nabla \bv|^2 \phi^2 d\y ds \\
&\lesssim \iint |\bv|^2 (\partial_t \phi^2 + \Delta \phi^2) d\y ds
+ \iint |\bv|^2 \bv \phi \cdot \nabla \phi d\y ds \\
&\quad + 2 \iint p \bv \phi \cdot \nabla \phi d\y ds \\ 
&=: I + II + III.
\end{split}
\end{equation}
We shall estimate each term on the right. 
By the Jensen inequality 
\begin{equation}
\label{E32}
\begin{split}
I
&= r^{-2} \iint |\bv|^2 d\y ds 
\lesssim r \int_{-4r^2}^0 \fint_{B(2r)} |\bv|^2 d\y ds \\
&\lesssim r \int_{-4r^2}^0 \Big(\fint_{B(2r)} |\bv|^6 d\y\Big)^{1/3} ds 
\lesssim r G(2r).
\end{split}
\end{equation}
Since $\nabla \cdot \bv=0$, we have 
\[
II = \iint (|\bv|^2-|\inn{\bv}_{B(2r)}|^2) \bv \phi \cdot \nabla \phi d\y ds.
\]
Using the H\"older inequality and then applying the Sobolev--Poincar\'e inequality, we obtain that 
\begin{align*}
II
&\lesssim r^{-1} \iint |\bv - \inn{\bv}_{B(2r)}| |\bv + \inn{\bv}_{B(2r)}| |\bv| \phi d\y ds \\
&\lesssim r^{-1/2} \int_{-4r^2}^0 
\Big(\int_{B(2r)} |\bv - \inn{\bv}_{B(2r)}|^6 d\y\Big)^{1/6} \\
&\qquad \qquad \qquad \qquad \times \Big(\int_{B(2r)} |\bv + \inn{\bv}_{B(2r)}|^6 d\y\Big)^{1/6} 
\Big(\int |\bv|^2 \phi^2 d\y\Big)^{1/2} ds \\
&\lesssim r^{-1/2} \sup_s \Big(\int |\bv|^2 \phi^2 d\y\Big)^{1/2} \int_{-4r^2}^0 \Big(\int_{B(2r)} |\nabla\bv|^2 d\y\Big)^{1/2} \Big(\int_{B(2r)} |\bv|^6 d\y\Big)^{1/6} ds \\
&\lesssim r^{1/2} \sup_s \Big(\int |\bv|^2 \phi^2 d\y\Big)^{1/2} E(2r)^{1/2} G(2r)^{1/2}.
\end{align*}
By the Young inequality we have for some $C>0$ for all $\delta>0$
\begin{equation}
\label{E33}
II \le \delta \sup_s \int |\bv|^2 \phi^2 d\y + \frac{Cr}{4\delta} E(2r) G(2r).
\end{equation}
H\"older's inequality gives 
\begin{align*}
III 
&= \iint p \bv \phi \cdot \nabla \phi d\y ds 
\lesssim r^{-1} \iint |p-c| |\bv| \phi d\y ds \\
&\lesssim r^{-1/2} \int_{-4r^2}^0 \Big(\int_{B(2r)} |p-c|^3 d\y\Big)^{1/3} \Big(\int |\bv|^2 \phi^2 d\y\Big)^{1/2} ds \\
&\lesssim r^{1/2} \sup_s \Big(\int |\bv|^2 \phi^2 d\y\Big)^{1/2} P(2r)^{1/2}.
\end{align*}
By the Young inequality we have for some $C>0$ for all $\delta>0$
\begin{equation}
\label{E34}
III \le \delta \sup_s \int |\bv|^2 \phi^2 d\y + \frac{Cr}{4\delta} P(2r).
\end{equation}
Combining \eqref{E31}--\eqref{E34} with a fixed small number $\delta$, we get the result.

\end{proof}

\begin{rem}
\label{R31}
If $\bbf \neq 0$, then we have 
for $0 < r \le 1$
\[A(r) + E(r) \lesssim [1+E(2r)] G(2r) + P(2r) + F(2r)\]
where 
\[F(r) = \left(\int_{-r^2}^0 \Big(\int_{B(r)} |\bbf|^2 d\y\Big)^{2/3} ds\right)^{3/2}.\]
Indeed, H\"older's inequality gives 
\begin{align*}
\iint \bbf \cdot \bv \phi^2 d\y ds 
&\lesssim \int_{-4r^2}^0 \Big(\int_{B(2r)} |\bbf|^2 d\y\Big)^{1/2} \Big(\int |\bv|^2 \phi^2 d\y\Big)^{1/2} ds \\
&\lesssim \sup_s \Big(\int |\bv|^2 \phi^2 d\y\Big)^{1/2} \int_{-4r^2}^0 \Big(\int_{B(2r)} |\bbf|^2 d\y\Big)^{1/2} ds \\
&\lesssim r^{1/2} \sup_s \Big(\int |\bv|^2 \phi^2 d\y\Big)^{1/2} F(2r)^{1/2}.
\end{align*}
By the Young inequality we have for some $C>0$ for all $\delta>0$
\[\iint \bbf \cdot \bv \phi^2 d\y ds 
\le \delta \sup_s \int |\bv|^2 \phi^2 d\y + \frac{Cr}{4\delta} F(2r).\]
As in the proof of the previous lemma, we can absorb the first term on the right by choosing small $\delta$.
We notice that $F(r)\to0$ as $r\to0$.
\end{rem}

\begin{rem}
The implied constants of the estimates in this section are all absolute.
\end{rem}

\section{Pressure inequalities}
\label{S4}

In this section we present pressure inequalities, Lemma \ref{L41} and Lemma \ref{L44}, which are used to complete iteration schemes.

\begin{lem}
[pressure inequality I]
\label{L41}
For $0 < r \le 1$ and $0 < \theta < 1/4$
\[D(\theta r) \lesssim \theta D(r) + \theta^{-2} \widetilde{C}(r).\]
\end{lem}

\begin{proof}
We may assume $r=1$.
In the sense of distributions we have 
\[-\Delta p = \partial_j\partial_k(\bv_j\bv_k).\] 
Let $\widetilde{\bv}=\bv-\inn{\bv}_{B(1)}$ and let $p_1$ satisfy the equation 
\[
-\Delta p_1 = \partial_j \partial_k (\widetilde{\bv}_j\widetilde{\bv}_k\phi)
\] 
where $\phi$ is a cutoff function which equals 1 in $Q(1/2)$ and vanishes outside of $Q(1)$.
By the Calderon--Zygmund inequality 
\begin{equation}
\label{E41}
\theta^{-2} \iint_{Q(\theta)} |p_1|^{3/2} d\y ds
\lesssim \theta^{-2} \widetilde{C}(r).
\end{equation}
Since $p_2 := p-p_1$ is harmonic in $B(1/2)$, we have by the mean value property 
\begin{equation}
\label{E42}
\begin{split}
\theta^{-2} \iint_{Q(\theta)} |p_2|^{3/2} d\y ds
&\lesssim \theta \iint_{Q(1/2)} |p_2|^{3/2} d\y ds \\
&\lesssim \theta D(1) + \theta \iint_{Q(1)} |p_1|^{3/2} d\y ds
\end{split}
\end{equation}
Since we have 
\[D(\theta) \lesssim \theta^{-2} \iint_{Q(\theta)} |p_1|^{3/2} + |p_2|^{3/2} d\y ds,\]
combining the two estimates \eqref{E41} and \eqref{E42} yields the result.

\end{proof}

Now, we recall a decomposition of Lebesgue spaces.

\begin{defn}
For $1< p < \infty$ define 
\begin{align*}
\cA^p(\Omega) &= \Set{\Delta v : v \in W^{2,\, p}_{ 0}(\Omega)}, \\
\cB^p(\Omega) &= \Set{ p_h \in L^p(\Omega)\cap C^\infty(\Omega) : \Delta p_h=0}.
\end{align*}
\end{defn}

\begin{lem}
\label{L42}
Let $1< p < \infty$ and $\Omega \subset \R^n$ be a bounded $C^2$-domain.
Then 
\[L^p(\Omega) = \cA^p(\Omega) \oplus \cB^p(\Omega).\]
\end{lem}

\begin{proof}
The proof can be found in \cite{Si72}. 
\end{proof}

\begin{rem}
\label{R41}
Denote $L^p_0(\Omega) = \set{f \in L^p(\Omega) : \inn{f}_{\Omega}=0}$ and 
\[\cB^p_0(\Omega) = \cB^p(\Omega) \cap L^p_0(\Omega).\] 
Since $\cA^p(\Omega) \subset L^p_0(\Omega)$,  
Lemma \ref{L42} implies that 
\[
L^p_0(\Omega)= \cA^p(\Omega) \oplus \cB^p_0(\Omega). 
\]
\end{rem}

\begin{lem}
\label{L43}
For $1<s<\infty$ the operator $T_s: \cA^s(\Omega) \rightarrow W^{-2,s}(\Omega)$ defined by 
\[\inn{T_s p_0, v} = \int_{\Omega} p_0 \Delta v,\quad  v\in W^{2,s'}_0(\Omega).\]
is an isomorphism. 
\end{lem} 

\begin{proof}
Let $ p_0 \in \cA^s(\Omega)$ and set 
\[q = |p_0|^{s-2} p_0 \in L^{s'}(\Omega).\] 
By Lemma \ref{L42} there exist unique $q_0 \in \cA^{s'}(\Omega)$ and $q_h \in \cB^{s'}(\Omega)$ such that 
\[q = q_0 + q_h.\] 
In particular, $q_0 = \Delta v_0$ for some $v_0\in W^{2,s'}_0(\Omega)$. 
Hence 
\begin{align*}
\norm{p_0}^s_{L^s(\Omega)} 
&= \int_{\Omega} p_0 q = \int_{\Omega} p_0 \Delta v_0 
\le \norm{T_s p_0}_{W^{-2,s}(\Omega)} \norm{v_0}_{W^{2,s'}_0(\Omega)} \\
&\lesssim \norm{T_s p_0}_{W^{-2,s}(\Omega)} \norm{q_0}_{L^{s'}(\Omega)}  
\lesssim \norm{T_s p_0}_{W^{-2,s}(\Omega)} \norm{p_0}^{s-1}_{L^s(\Omega)}.      
\end{align*}
This implies that 
\[\norm{p_0}_{L^s(\Omega)} \lesssim \norm{T_s p_0}_{W^{-2,s}(\Omega)}\]
and the operator $ T_s$ has closed range.  
Furthermore, Lemma \ref{L42} implies also that if $T_s p_0=0$, then $p_0 \in \cA^s(\Omega) \cap \cB^s(\Omega) = \set{0}$. 
Hence $T_s$ is injective and the result follows from the closed range theorem.

\end{proof}

\begin{rem}
\label{R42}
\begin{enumerate}
\item
Let $\bbf \in \bL^s(\Omega ; \R^{n\times n})$, $1<s<\infty$. 
Then by Lemma \ref{L43} there exists a unique $ p_0 \in \cA^s(\Omega)$ such that 
\begin{equation}
\label{E44}
\Delta p_0 = \nabla \cdot \nabla \cdot \bbf
\end{equation}  
in $\Omega$ in the sense of distributions.
Morevoer, there holds the estimate 
\begin{equation}
\label{E45}
\norm{p_0}_{L^s(\Omega)} \lesssim \norm{\bbf}_{L^s(\Omega)}.
\end{equation}
\item
Let $\bg \in \bL^s(\Omega; \R^{n} )$, $1<s< n$. 
Then by means of Sobolev's embedding theorem 
$\nabla \cdot \bg \in W^{-1,s}(\Omega) \hookrightarrow W^{-2,\, s^{\ast}}(\Omega)$ 
where $s^{\ast} = ns/(n-s)$. 
Thus, there exists a unique $p_0 \in \cA^{ s^{\ast}}(\Omega)$ such that 
\begin{equation}
\label{E46}
\Delta p_0 = \nabla \cdot \bg
\end{equation} 
in $\Omega$ in the sense of distributions. 
By the definition of $\cA^{ s^{\ast}}(\Omega)$ there exist 
$v_0 \in W^{2,s^{\ast}}_0(\Omega)$ with $ \Delta v_0 = p_0$, and there holds $\Delta ^2 v_0 = \nabla \cdot \bg $ in $\Omega$ in the sense of distributions. By means of 
elliptic regularity we find $ v_0 \in W^{3,s}(\Omega)$ together with the estimate 
\begin{equation}
\label{E47}
\norm{\nabla p_0}_{L^s(\Omega)} \lesssim \norm{v_0}_{W^{3,s}(\Omega)} 
\lesssim \norm{\bg}_{L^s(\Omega)}. 
\end{equation} 
\item
Let $p\in L^s(\Omega)$. 
In view of Lemma \ref{L42} we have $p = p_0 + p_h$ with unique 
$p_0 \in \cA^s(\Omega)$ and $p_h \in \cB^s(\Omega)$. 
Observing that $p- \inn{p}_{\Omega} = p_0 + (p_h- \inn{p_h}_{\Omega})$ 
and appealing to Remark \ref{R41} it follows that  
\begin{equation}
\label{E48}
\norm{p_h-\inn{p_h}_{\Omega}}_{L^s(\Omega)} 
\lesssim \norm{p-\inn{p}_{\Omega}}_{L^s(\Omega)}. 
\end{equation}
\item
The implied constant in \eqref{E45}, \eqref{E47} and \eqref{E48} depend only on $s$ and $\Omega$.
When $\Omega$ equals a ball, these constants depend on $s$ but not on the radius of the ball.  
\end{enumerate}
\end{rem}

\begin{lem}
[pressure inequality II]
\label{L44}
For $0 < r \le 1$ and $0< \theta \le 1/4$
\[P(2\theta r) \lesssim \theta^2 P(r) + \theta^{-2} E(r)^2 + \theta^{-2} F(r).\]
\end{lem}

\begin{proof}
We may assume $r=1$ and denote $B=B(1)$ and $Q=Q(1)$.
By Lemma \ref{L42} we may decompose for a. e. $t \in I_{R}$ 
\[p = p_0 + p_h\] 
where $p_0 \in \cA^3(B)$ 
and $p_h \in \cB^3(B)$ is harmonic. 
By Remark \ref{R42} we may decompose 
\[p_0 = p_{01} + p_{02}\] 
where $ p_{01}
\in \cA^3(B)$ is the unique weak solution to 
\[\Delta p_{01} 
= - \nabla \cdot \nabla \cdot ((\bv - \inn{\bv}_B) \otimes (\bv - \inn{\bv}_B))\]
in $B$ in the sense of distributions, 
while $p_{02} \in \cA^3(B)$ is the unique weak solution to 
\[\Delta p_{02} = \nabla \cdot \bbf\]
in $B$ in the sense of distributions for a. e. $ t\in I(r) := (-r^2,0)$. 

By the aid of \eqref{E45} and \eqref{E46} along with Sobolev-Poincar\'e's inequality, we find that for a. e. $ t\in I_{R}$
\begin{align*}
\norm{p_{01}(t)}_{L^3(B)} 
&\lesssim \norm{\bv(t) - \inn{\bv}_B(t)}^2_{\bL^6(B)} 
\lesssim \norm{\nabla \bv(t)}^2_{\bL^2(B)}, \\
\norm{p_{02}(t)}_{L^3(B)} 
&\lesssim \norm{\nabla p_{02}(t)}_{\bL^{3/2}(B)}
\lesssim \norm{\bbf(t)}_{\bL^{3/2}(B)}. 
\end{align*}
Integrating in time, we get 
\begin{align}
\label{E410}
\int_I \norm{p_{01}}_{L^3(B)} ds 
&\lesssim \int_I \norm{\nabla \bv}^2_{\bL^2(B)} ds = E(1), \\
\label{E411}
\int_I \norm{p_{02}}_{L^3(B)} ds 
&\lesssim \int_I \norm{\bbf }_{\bL^{3/2}(B)} ds = F(1)^{1/2}.  
\end{align}

On the other hand, employing \eqref{E48}, 
we see that $p_h - \inn{p_h}_B \in L^1(I;L^3(B))$ and 
\[\int_I \norm{p_h - \inn{p_h}_B}_{L^3(B)} ds
\lesssim \int_I \norm{p - \inn{p_h}_B}_{L^3(B)} ds.\]
Applying the Poincar\'e-type inequality and using the mean value property of harmonic functions, we obtain that 
\begin{equation}
\label{E413}
\begin{split}
\int_{I(2\theta)} \norm{p_h - \inn{p_h}_{ B_{ \theta  R}}}_{L^3(B(2\theta))} ds 
&\lesssim \theta^2 \int_{I(1/2)} \norm{\nabla p_h}_{\bL^\infty(B(1/2))} ds \\
&\lesssim \theta^2 \int_I \norm{p - \inn{p}_B}_{L^3(B)} ds.
\end{split}
\end{equation}
Combining \eqref{E410}, \eqref{E411}, and \eqref{E413}, we get 
\begin{align*}
P(2\theta)^{1/2} 
&\lesssim \theta^{-1} \int_{I(2\theta)} \norm{p - \inn{p}_{B(\theta)}}_{L^3(B(2\theta))} ds \\
&\lesssim \theta^{-1} \int_{I(2\theta)} \norm{p_h - \inn{p_h}_{B(\theta)}}_{L^3(B(2\theta))} ds \\
&\quad + \theta^{-1} \int_I \norm{p_{01}}_{L^3(B)} ds
+ \theta^{-1} \int_I \norm{p_{02}}_{L^3(B)} ds\\
&\lesssim \theta \int_I \norm{p - \inn{p}_B}_{L^3(B)} ds
+ \theta^{-1} E(1) + \theta^{-1} F(1)^{1/2} 
\end{align*}
and the result follows.

\end{proof}

\begin{rem}
The implied constants of the estimates in this section are all absolute.
\end{rem}

\section{Interpolation inequalities}
\label{S5}

In this section we give a few interpolation inequalities.
We shall use one more scaled functional,
\[\widetilde{C}(r) = r^{-2} \iint_{Q(r)} |\bv-\inn{\bv}_{B(r)}|^3 d\y ds.\]

\begin{lem}
\label{L51}
For $0 < r \le 1$ and $0 < \theta \le 1$
\[C(\theta r) \lesssim \theta C(r) + \theta^{-2} \widetilde{C}(r)\]
and 
\begin{equation}
\label{E51} 
C(\theta r) \lesssim \theta^3 A(r)^{3/2} + \theta^{-2} \widetilde{C}(r).
\end{equation}
\end{lem}

\begin{proof}
We may assume $r=1$ and denote $B=B(1)$ and $\inn{\bv}_B = \fint_B \bv d\y$.
By subtracting the average $\inn{\bv}_B$ we have 
\[
\int_{B(\theta)} |\bv|^3 d\y \lesssim \theta^3 |\inn{\bv}_B|^3 + \int_{B(\theta)} |\bv-\inn{\bv}_B|^3 d\y.
\]
Integrating in time and using Jensen's inequality we get the result.

\end{proof}

\begin{lem} 
[interpolation inequality I]
\label{L52}
Let 
\begin{equation}
\label{E52}
\frac{9}{5} \le q \le 2, \quad \frac{3-q}{5q-6} \le k \le \frac{3-q}{3}.
\end{equation}
Then for $0<r\le1$ 
\begin{equation}
\label{E53}
\widetilde{C}(r) \lesssim A(r)^{(9-3q-3qk)/(6-2q)} E_q(r)^{3k/(3-q)}.
\end{equation}
\end{lem}

\begin{proof}
By scaling we may assume $r=1$ and denote $B=B(1)$.
By the Sobolev-Poincar\'e inequality 
\begin{align*}
\int_B |\bv-\inn{\bv}_B|^3 d\y
&\lesssim \Big(\int_B |\bv|^2 d\y\Big)^{(3-kq^*)/2} \Big(\int_B |\bv-\inn{\bv}_B|^{q^*} d\y\Big)^k \\
&\lesssim A(1)^{(3-kq^*)/2} \Big(\int_B |\nabla \bv|^q d\y\Big)^{kq^*/q}
\end{align*}
where $q^*=3q/(3-q)$.
Note that from \eqref{E52} we have $0<(3-kq^*)/2<1$, $0<k<1$, and 
\[
0 < (3-kq^*)/2 + k \le 1, \quad 0 < kq^*/q \le 1.
\]
By the Jensen inequality 
\begin{align*}
&\int_{-1}^0 \int_{B} |\bv-(\bv)_B|^3 d\y ds \\
&\lesssim A(1)^{(3-kq^*)/2} \int_{-1}^0 \Big(\int_B |\nabla \bv|^q d\y\Big)^{kq^*/q} ds \\
&\lesssim A(1)^{(3-kq^*)/2} E_q(1)^{kq^*/q}.
\end{align*}
A calculation shows 
\[(3-kq^*)/2=(9-3q-3qk)/(6-2q)\]
and 
\[kq^*/q=3k/(3-q).\]

\end{proof}

\begin{rem}
\label{R51}
If we choose $q=2$ and $k=1/4$, then the estimate \eqref{E53} becomes the well-known estimate 
\[
\widetilde{C}(r) \lesssim A(r)^{3/4} E(r)^{3/4}.
\]
If we choose $k=(3-q)/3$, then the estimate \eqref{E53} becomes
\begin{equation}
\label{E54}
\widetilde{C}(r) \lesssim A(r)^{(3-q)/2} E_q(r).
\end{equation}
\end{rem}

\begin{lem}
\label{L53}
Let
\[X(r) := C(r) + D(r).\]
If $\frac{9}{5} \le q \le 2$ and $\frac{3-q}{5q-6} \le k \le \frac{3-q}{3}$, then for $0 < r \le 1$ and $0 < \theta < \frac{1}{4}$
\[X(\theta r) \lesssim \theta X(r) + \theta^{-2} A(r)^{(9-3q-3qk)/(6-2q)} E_q(r)^{3k/(3-q)}.\]
\end{lem}

\begin{proof}
It follows from combining Lemma \ref{L51}, \ref{L41}, and \ref{L52}.

\end{proof}

\begin{lem}
[interpolation inequality II]
\label{L54}
For $0 < r \le 1$ and $0 < \theta \le 1$

\[
G(\theta r) \lesssim  \theta^{-1} E(r) + \theta^2 A(r). 
\]
\end{lem}

\begin{proof}
We may assume $r=1$ and denote $B=B(1)$ and $\inn{\bv}_B = \fint_B \bv d\y$.
By the Sobolev-Poincar\'e inequality 
\begin{align*}
\int_{B(\theta)} |\bv|^6 d\y
&\lesssim \int_{B(\theta)} |\bv-\inn{\bv}_B|^6 d\y + \int_{B(\theta)} |\inn{\bv}_B|^6 d\y \\
&\lesssim \Big(\int_B |\nabla\bv|^2 d\y\Big)^3 + (\theta r)^3 |\inn{\bv}_B|^6.
\end{align*}
Thus, we have
\begin{align*}
G(\theta)
&= \theta^{-1} \int_{-\theta^2}^0 \Big(\int_{B(\theta)} |\bv|^6 d\y\Big)^{1/3} ds \\
&\lesssim \theta^{-1} E(r) + \int_{-\theta^2}^0 |\inn{\bv}_B|^2 ds,
\end{align*}
and the result follows.

\end{proof}

\begin{lem}
\label{L55}
\[\Big(r^{-2} \iint_{Q(r)} |\bv|^3 d\y ds\Big)^{2/3} \lesssim A(r) + E(r).\]
\end{lem}

\begin{proof}
By scaling we may assume $r=1$ and denote $B=B(1)$ and $Q=Q(1)$.
By the H\"older inequality
\[
\iint_Q |\bv|^3 d\y ds 
\le \int_{-1}^0 \Big(\int_B |\bv|^2 d\y\Big)^{1/2} 
\Big(\int_B |\bv|^6 d\y\Big)^{1/3} ds.
\]
By the Young inequality 
\[\Big(\iint_Q |\bv|^3 d\y ds\Big)^{2/3} \lesssim A(1)^{1/3} G(1)^{2/3} \le A(1) + G(1).\]
By Lemma \ref{L54} with $\theta=1$ we get the result.

\end{proof}

\begin{rem}
The implied constants of the estimates in this section are all absolute.
\end{rem}

\section{Control of local kinetic energy and pressure}
\label{S6}

The aim of this section is to prove that the scaled quantities of local kinetic energy and pressure are controlled by the velocity gradient.

\begin{lem}
\label{L61}
Let $9/5 \le q \le 2$.
There exists an absolute positive constant $\gamma$ such that if $1<\overline{E}_q<\infty$, then 
\begin{equation}
\label{E61}
\limsup_{r\to0} [A(r)+D(r)] \le \gamma \overline{E}_q^{2/(q-1)}.
\end{equation}
\end{lem}

\begin{rem}
We assume $\overline{E}_q>1$ for convenience.
Indeed, we may consider the case that $\overline{E}_q$ has a positive lower bound because of the criterion \eqref{E13}.
\end{rem}

\begin{proof}
Fix $q$ and denote $M=\overline{E}_q$.
There is $R<1$ such that for all $0 < r < R$
\[
E_q(r) \le 2M.
\]
From the local energy inequality I in Section \ref{S3}, we have for $0 < r < R$ and $0 < \theta \le 1$
\[A(\theta r) \lesssim 1 + X(2\theta r)\]
where $X(r) = C(r) + D(r)$.
If we set
\[Y(r) := A(r)+X(r),\]
then, by using the trivial estimate $X(\theta r) \le 4 X(2\theta r)$, we get 
\begin{equation}
\label{E62}
Y(\theta r) \lesssim 1 + X(2\theta r).
\end{equation}
Using Lemma \ref{L53} with $k = (3-q)/3$ and then applying Young's inequality, we obtain that for $0 < r < R$ and $0 < \theta < 1/4$
\begin{equation}
\label{E63}
\begin{split}
X(2\theta r) 
&\lesssim \theta X(r) + \theta^{-2} A(r)^{(3-q)/2} M \\
&\lesssim \theta Y(r) + \theta^{-(7-q)/(q-1)} M^{2/(q-1)}.
\end{split}
\end{equation}
Thus, combining \eqref{E62} and \eqref{E63} yields that for some positive constant $\beta\ge2$ 
\begin{align*}
Y(\theta r) 
&\le \beta \theta Y(r) + \beta \theta^{-(7-q)/(q-1)} M^{2/(q-1)} + \beta \\
&\le \beta \theta Y(r) + 2\beta \theta^{-(7-q)/(q-1)} M^{2/(q-1)}.
\end{align*}
If we fix $\theta = (2\beta)^{-1}$, then the last inequality becomes 
\[Y(\theta r) \le \frac{1}{2} Y(r) + (2\beta)^{6/(q-1)} M^{2/(q-1)}.\]
By the standard iteration argument we get 
\[
\limsup_{r\to0} Y(r) \le \gamma M^{2/(q-1)}
\]
where $\gamma = 2(2\beta)^{6/(q-1)}$.
This completes the proof.

\end{proof}

\begin{lem}
\label{L62}
There exists an absolute positive constant $\gamma$ such that if $\overline{E}<\infty$, then 
\[\limsup_{r\to0} P(r) \le \gamma \overline{E}^2.\]
\end{lem}

\begin{proof}
From Lemma \ref{L44} we have for all $r<1$ and $0< \theta \le 1/4$
\[P(2\theta r) \lesssim \theta^2 P(r) 
+ \theta^{-2} E(r)^2 + \theta^{-2} F(r).\]
Since $\lim_{r\to0} F(r)=0$, we initially start from a small number $r=R$ and then perform a standard iteration argument to get the result.

\end{proof}

\section{Proof of Theorem \ref{T1}}
\label{S7}

Fix $q$ and denote 
\[M = \overline{E}_q \quad \text{ and } \quad m = \underline{E}_q.\]
Suppose $1 < M < \infty$ for convenience.
Lemma \ref{L61} implies that there is a positive number $R$ such that for all $0<r\le R$
\begin{equation}
\label{E71}
A(r) \lesssim M^{2/(q-1)} \quad \text{ and } \quad D(r) \lesssim M^{2/(q-1)}.\end{equation}
On the other hand, there exists a sequence of positive numbers $r_n$ such that $r_n < R$ and  
\[
\lim_{n\to\infty} r_n = 0 \quad \text{ and } \quad \lim_{n\to\infty} E_q(r_n) = m. 
\]
Combining \eqref{E51} and \eqref{E54}, we have for all $n$ and $0 < \theta \le 1$
\[
C(\theta r_n) \lesssim \theta^3 A(r_n)^{3/2} + \theta^{-2} A(r_n)^{(3-q)/2} E_q(r_n).
\]
Hence from \eqref{E71} we obtain that for some $\beta>0$ 
\[
C(\theta r_n) \le \beta \theta^3 M^{3/(q-1)} + \beta \theta^{-2} M^{(3-q)/(q-1)} E_q(r_n).
\]

If $0<m$, then we take $\theta=[M^{-q/(q-1)} m]^{1/5}$ so that 
\begin{align*}
C(\theta r_n) 
&\le \beta \theta^3 M^{3/(q-1)} + \beta \theta^{-2} M^{(3-q)/(q-1)} E_q(r_n) \\
&\le \beta \Big(M^{(5-q)/(q-1)} m\Big)^{3/5} \Big(1 + m^{-1} E_q(r_n)\Big).
\end{align*}
Since 
\[
\lim_{n\to\infty} m^{-1} E_q(r_n) = 1,
\]
we have for all large $n$
\[
C(\theta r_n) \le 3 \beta \epsilon^{3/5}.
\]
If $\epsilon$ is small, then we take $R=\theta r_N$ and a large natural number $N$ so that $\z$ is a regular point.

If $m=0$, then theorem is trivially true.
Indeed, we can choose $\theta$ so that $\beta \theta^3 M^{3/(q-1)}$ is small enough and 
\[
\lim_{n\to\infty} \beta \theta^{-2} M^{(3-q)/(q-1)} E_q(r_n) = 0.
\]
Therefore $\z$ is a regular point.
This completes the proof of Theorem \ref{T1}.

\section{Proof of Theorem \ref{T2}}
\label{S8}

Finally, we give the proof of Theorem \ref{T2}.
From Remark \ref{R31} we have for all $2\theta r<R$ and $0< \theta < 1/4$
\[A(\theta r) + E(\theta r) 
\lesssim [1+E(2\theta r)] G(2\theta r) + P(2\theta r) + F(R)\]
where $R$ will be determined later.
From Lemma \ref{L44}, we have for $0 < \theta < 1/4$
\[P(2\theta r) \lesssim \theta^2 P(r) + \theta^{-2} E(r)^2 + \theta^{-2} F(R).\]
From Lemma \ref{L54} 
\[G(2\theta r) \lesssim  \theta^{-1} E(r) + \theta^2 A(r).\]
We also have 
\[E(2\theta r) \le (2\theta)^{-1} E(r)\] 
by the definition.
Combining all the above estimates, we conclude that for $2\theta r<R$ and $0< \theta < 1/4$
\begin{equation}
\label{E81}
\begin{split}
&A(\theta r) + E(\theta r) \\
&\lesssim \theta A(r) E(r) + \theta^2 A(r) + \theta^2 P(r) + \theta^{-1} E(r) + \theta^{-2} E(r)^2 + \theta^{-2} F(R).
\end{split}
\end{equation}

Let us denote 
\[M = \overline{E} \quad \text{ and } \quad m = \underline{E}.\]
If $m=0$, then theorem is trivially true.
We may consider the case $0 < m$ and $1 \le M < \infty$.
Lemma \ref{L61} with $q=2$ implies that there is a positive number $R_1$ such that for all $0<r\le R_1$
\begin{equation}
\label{E82}
A(r) \lesssim M^2.
\end{equation}
Lemma \ref{L62} implies that there is a positive number $R_2$ such that for all $0<r\le R_2$
\begin{equation}
\label{E83}
P(r) \lesssim M^2.
\end{equation}
Since $\lim_{r\to0} F(r)=0$, there is a positive number $R_3$ such that for all $0<r\le R_3$
\begin{equation}
\label{E84}
F(r) \le M^{-2} \epsilon^2.
\end{equation}
We also have for some $R_4$ and for all $0<r\le R_4$
\[E(r) \le 2M.\]
We can take 
\[R = \min\set{R_1,R_2,R_3,R_4}\]
and fix a sequence $r_n$ such that $r_n < R$, 
\[\lim_{n\to\infty} r_n = 0 \quad \text{ and } \quad \lim_{n\to\infty} E(r_n) = m.\]

Combining \eqref{E81}--\eqref{E84}, we have for all sufficiently large $n$ and for all $0< \theta < 1/4$
\begin{align*}
&A(\theta r_n) + E(\theta r_n) \\
&\lesssim \theta M^2 E(r_n) + \theta^2 M^2 + \theta^{-1} E(r_n) + \theta^{-2} E(r_n)^2 + \theta^{-2} M^{-2} \epsilon^2 \\
&\lesssim \theta M^2 m + \theta^2 M^2 + \theta^{-1} m + \theta^{-2} m^2 + \theta^{-2} M^{-2} \epsilon^2.
\end{align*}
Since $Mm < \epsilon$ and $\epsilon<1/16$, we can take $\theta = \epsilon^{1/2} M^{-1} < 1/4$ so that the above estimate becomes 
\begin{equation}
\label{E85}
A(\theta r_n) + E(\theta r_n) 
\lesssim \epsilon^{3/2} + \epsilon + \epsilon^{1/2}
\lesssim \epsilon^{1/2}.
\end{equation}
As it has been proved in \cite{Wo10} there exists an absolute constant $\epsilon$ such that if $D(r) \le \epsilon$ that $z$ is a regular point (cf. \cite{Wo15}).
This together with Lemma \ref{L55} shows that there exists a positive constant $\epsilon$ such that $\z$ is a regular point if for some $r>0$
\[A(r) + E(r) < \epsilon.\] 
Due to \eqref{E83} and \eqref{E85}, we conclude that the reference point $z$ is regular for the case $m>0$.
This completes the proof of Theorem \ref{T2}.

\section*{Acknowledgement}

H. J. Choe has been supported by the National Reserch Foundation of Korea(NRF) grant, 
funded by the Korea government(MSIP) (No. 20151009350).
J. Wolf has been supported by the German Research Foundation (DFG) through the project WO1988/1-1; 612414.  
M. Yang has been supported by the National Research Foundation of Korea(NRF) grant funded by the Korea government(MSIP) (No. 2016R1C1B2015731).

\end{document}